\documentclass[10pt]{amsart}
\usepackage{txfonts}
\usepackage[T1]{fontenc}
\usepackage{fancyhdr}

\usepackage{amsmath,amsthm,amssymb,hyperref, mdframed}
\usepackage{mathrsfs} 
\usepackage{tikz}
\usepackage[all]{xy}
\usetikzlibrary{decorations.markings}
\usetikzlibrary{backgrounds,shapes}
\usepackage{subfig}
\usetikzlibrary{shapes,arrows}
\usetikzlibrary{shapes.geometric, arrows}
\usetikzlibrary{decorations.markings}
\usepackage{subfig}
\usepackage{tabularx}
\usepackage{enumitem}
\usepackage{amsmath,amsthm,amscd}
\usepackage{amssymb,amsfonts}
\usepackage{fancyhdr} 
\usepackage{pgf,tikz}
\usepackage{caption}
\usepackage{tikz-cd}
\usepackage{tikzscale}
\usetikzlibrary{patterns}
\usepackage{rotating}
\usepackage{xcolor}
\usepackage{lscape}
\usepackage{array} 
\usepackage{tabularx}
\usepackage{multirow}
\usepackage{multicol}
\usepackage{booktabs} 
\usepackage{caption}
\usepackage{xspace}
\usepackage{xfrac}
\usepackage{blindtext}
\usepackage[ruled,vlined]{algorithm2e}

\usetikzlibrary{decorations.markings}
\usetikzlibrary{patterns}
\usetikzlibrary{calc}
\usetikzlibrary{shapes.geometric}

\usepackage{color}
\usepackage{hyperref}
\hypersetup{
    colorlinks=true, 
    linktoc=all,     
    linkcolor=blue,  
    citecolor=blue,
    filecolor=blue,
    urlcolor=blue
}

\definecolor{aliceblue}{rgb}{0.9, 0.95, 1.0}

\numberwithin{equation}{section}

\newcommand\Z{{\mathbb Z}}

\newcommand{\pslc}{{\mathrm{PSL}(2,\mathbb{C})}}

\newcommand{\pslr}{{\mathrm{PSL}(2,\mathbb{R})}}

\makeatletter
\def\namedlabel#1#2{\begingroup
   \def\@currentlabel{#2}%
   \label{#1}\endgroup
}
\makeatother

\theoremstyle{plain}                    
\newtheorem{thm}{Theorem}[section]

\newtheorem{lem}[thm]{Lemma}

\newtheorem*{thmnn}{Theorem}

\theoremstyle{definition}
\newtheorem{defn}[thm]{Definition}

\newtheorem*{coex}{Counterexample}
\newtheorem*{qs}{Question}

\begin{document}

\title[Counterexamples to the simple loop conjecture in higher-dimension]{Counterexamples to the simple loop conjecture in higher-dimension}

\author{Gianluca Faraco}
\address[Gianluca Faraco]{Dipartimento di Matematica e Applicazioni U5, Universita` degli Studi di Milano-Bicocca, Via Cozzi 55, 20125 Milano, Italy}
\email{gianluca.faraco@unimib.it}

\date{\today}
\dedicatory{}

\begin{abstract}
For every $g\ge 2$ and $n\ge4$, we provide an $n-$manifold $M$ and a continuous $2-$sided map $f\colon S\longrightarrow M$, where $S$ is a closed genus $g$ surface, such that no simple loop is contained in $\textnormal{ker}(\,f_*\,)$. This provides a counterexample to the the classical simple loop conjecture for surfaces to manifolds of dimensions at least four.
\end{abstract}

\maketitle
\tableofcontents

\section{Introduction}

\noindent For a closed surface $S$ of genus $g\ge2$, the so-called Simple Loop Conjecture stated that if a continuous map $f\colon S\longrightarrow \Sigma$ between closed surfaces induces a non-injective map of fundamental groups $f_*\colon\pi_1(S)\to\pi_1(\Sigma)$, then there exists an essential simple closed curve in $\textnormal{ker}(f_*)$. In \cite{GD}, Gabai provided a positive answer to this conjecture. It is natural to wonder what happens if the target is replaced with fundamental groups of higher dimensional manifolds and hence we may restate the Simple Loop Conjecture as

\medskip

\noindent \textbf{Simple Loop Conjecture.} \textit{Let $f\colon S\longrightarrow M$ be a continuous $2-$sided map from a closed surface $S$ to an $n-$dimensional manifold $M$ and let $f_*\colon\pi_1(S)\longrightarrow \pi_1(M)$ be the induced map on the fundamental group. If $\textnormal{ker}(f_*)$ is not trivial, then there is a simple closed essential loop in $\textnormal{ker}(f_*)$.}

\medskip

\noindent In \cite{HJ} Hass proved that the Simple Loop Conjecture holds for maps from a closed surface $S$ to a Seifert-fibered $3$-manifold $M$. A longstanding question has been whether result still holds if the target is replaced by any orientable $3$-manifold. Although this is known for graph-manifolds, see \cite{RW}, the question is essentially open. More recently, in \cite{ZD} the Simple loop conjecture has been verified for $3-$manifolds modelled on \textit{Sol} geometry. In the present note we aim to address the case of higher dimensional manifolds by showing the following

\begin{thmnn}
Let $S$ be a closed surface of genus $g\ge2$. For any $n\ge4$ there is a $n-$dimensional manifold $M$ and a continuous $2-$sided map $f\colon S\longrightarrow M$ such that the kernel $\textnormal{ker}(f_*)$ of the induced map $f_*\colon\pi_1(S)\longrightarrow \pi_1(M)$ is non-trivial and it does not contain any essential simple closed curve.
\end{thmnn}

\noindent The assumption $n\ge4$ plays an important role in the construction below. In fact, as we shall recall in Section \S\ref{sec:rp}, it is well-known fact that any finitely presented group $G$ can be realised as the fundamental group of some closed $n-$manifold $M$ for any $n\ge4$. Therefore, the basic idea is to find a finitely presented group $G$ and an epimorphism $\rho\colon\pi_1(S)\longrightarrow G$ whose kernel is non-trivial and such that no element in the kernel represents a simple closed essential loop in $S$. The desired counterexample follows by realising a continuous mapping $f\colon S\longrightarrow M$ such that $f_*=\rho$. Therefore the Simple Loop Conjecture, as stated, does not hold in dimension at least $4$. This leads to the following question which seems to be of general interest:

\begin{qs}
For a closed manifold $M$ of dimension $n\ge4$, provide necessary and sufficient conditions on $\pi_1(M)$ for the Simple Loop Conjecture to be verified.
\end{qs}

\smallskip

\noindent We conclude this introduction by mentioning a more challenging version of the Simple Loop Conjecture in which the target is replaced with a Lie group. In this direction, results have been provided for relevant Lie groups like $\pslr$ and $\pslc$. For these groups, it is known that the Simple Loop Conjecture fails. In fact, counterexamples for $\pslc$ were first provided by Cooper and Manning in \cite{CM}, and this answered a question of Minsky in \cite[Question 5.3]{Minsky}. Further counterexamples are described by Louder in \cite{Louder} and Calegari in \cite{Calegari}. For $\pslr$, in \cite{MaKat} provided explicit counterexamples; however they all come from indiscrite representation. It may be asked whether the Simple loop conjecture holds for discrete representations in $\pslr$. This is answered positively in a collaboration with Gupta, see \cite{FG3}.

\subsection*{Acknowledgments} The author wish to thank his friend Lorenzo Ruffoni for the comments which led to significant improvements of the exposition. He also wishes to thank Stefano Francaviglia and Subhojoy Gupta for their comments about this work.

\medskip

\section{$1-$sided vs $2-$sided maps}

\noindent Let $M$ be a connected manifold and let $\pi_1(M)$ denote its fundamental group. An \textit{orientation character} of $M$ is a representation $\chi_M\colon\pi_1(M)\longrightarrow \mathbb Z_2$ such that $\chi_M(\gamma)=1$ if and only if $\gamma$ acts on the universal cover of $M$ as an orientation-reversing homeomorphism. The character $\chi_M$ is trivial if and only if $M$ is orientable. See \cite{DK} for a reference.

\begin{defn}\label{def:sidesmap}
Let $M$ and $N$ be connected manifolds and let $\chi_M$ and $\chi_N$ be the respective orientation characters. A map $f\colon M\longrightarrow N$ is said to be $2-$\textit{sided} if 
\begin{equation}
    \chi_M= \chi_N\circ f_*
\end{equation}
holds, where $f_*\colon\pi_1(M)\longrightarrow\pi_1(N)$. Otherwise $f$ is said to be $1-$\textit{sided}. 
\end{defn}

\noindent In dimension $3$, the simple loop conjecture is known to fail for $1-$sided maps $f\colon S\longrightarrow M$ thus making the $2-$sided assumption necessary. Here is a counterexample that one may use to show that the same assumption is needed in higher dimension. 

\begin{coex}[in dimension $3$]
An easy case to think about is that of a $1-$sided embedded torus $\imath\colon S\cong\mathbb S^1\times\mathbb S^1\hookrightarrow \mathbb{RP}^2\times\mathbb S^1$. In fact $\imath_*\colon\pi_1(S)\cong\mathbb Z\times\mathbb Z\to \mathbb Z_2\times\mathbb Z$ has non-trivial kernel comprising all and only the curves corresponding to the pairs $\{(2k,0)\mid k\in\mathbb Z\}$. It is easy to see that these are all non-simple. This clearly makes the $2-$sided assumption in dimension $3$ necessary.
\end{coex}

\noindent For $n\ge4$, the counterexample above easily extends to $n$-dimensional manifolds. In fact, one may consider $M=\mathbb{RP}^2\times\mathbb S^1\times \mathbb S^{n-3}$ and hence the natural inclusion $\jmath\colon \mathbb{RP}^2\times\mathbb S^1\hookrightarrow M$.

\section{Finding a finitely presented group $G$}\label{sec:fpg}

\noindent Let $S$ be a closed surface of genus $g\ge2$. For a group $G$, a representation $\rho\colon \pi_1(S)\longrightarrow G$ has \textit{geometric kernel} if some nontrivial element in the kernel can be represented by a simple loop. Otherwise we shall say that $\rho$ has \textit{non-geometric kernel}. In this section we aim to realize a finite group $G$ -- and thence finitely presented -- and an epimorphism $\rho\colon\pi_1(S)\longrightarrow G$ with non-geometric kernel. This process is not new in literature and, indeed, it is attributed to Casson. We provide here a direct construction since we cannot find a direct reference. The basic idea relies on the following observation: If $p\colon\Sigma \longrightarrow S$ is a regular covering such that no essential simple closed curve on $S$ lifts to $\Sigma$ then the natural projection
\begin{equation}\label{eq:proj}
    \pi_1(S)\longrightarrow \frac{\pi_1(S)}{p_*\big(\pi_1(\Sigma)\big)}=G
\end{equation}
has non-geometric kernel.

\medskip

\noindent Given a closed genus $g$ surface $S$, let us consider the covering space $p\colon\Sigma'\longrightarrow S$ corresponding to the kernel of the epimorphism $\phi\colon\pi_1(S)\longrightarrow \textnormal{H}_1(S,\,\mathbb Z_2)$. Observe that simple non-separating loops on $S$ are generators of $\textnormal{H}_1(S,\,\mathbb Z_2)$ and hence \textit{do not lift} to $\Sigma'$. The crucial observation here is the following

\begin{lem}\label{lem:nonseplifts}
Every simple separating loop lifts to $\Sigma'$ and each of the preimages is a simple non-separating loop on $\Sigma'$.
\end{lem}


\begin{proof}
In the first place we notice that $p\colon\Sigma'\longrightarrow S$ is a $2^{2g}$ degree covering and hence $\Sigma'$ is a closed surface. Let $\gamma$ be a simple essential separating loop in $S$ \textit{i.e.} $S=S_1\,\cup_\gamma\,S_2$. We argue by contradiction and assume that some -- and hence any -- lift separates $\Sigma'$ into two sub-surfaces. Since $\deg(\,p\,)=2^{2g}$, there are $2^{2g}$ lifts, say $\gamma_1',\dots,\gamma_{2^{2g}}'$, of $\gamma$. All these curves being separating, the cut surface 
\begin{equation}
    \Sigma'\setminus \bigcup_{i=1}^{2^{2g}} \gamma_i'
\end{equation}
\noindent has $2^{2g}+1$ connected components. A few of these components, say $m$, project to $S_1$ and the remaining $2^{2g}+1-m$ components project to $S_2$. The covering projection $p$ restricts to each of these components to a cover of either $S_1$ or $S_2$. Let $\Sigma_1',\dots,\Sigma_m'$ denote the connected components of the cut surface that project to $S_1$ and denote by $\Sigma_{m+1}',\dots,\Sigma_{2^{2g}+1}'$ the remaining connected components -- of course they all project to $S_2$. Let $d_i$ be defined as the degree of $p$ restricted to the connected component $\Sigma_i'$, namely
\begin{equation}
    d_i=\deg\big(\, p_{|\,\Sigma_i'}\colon\Sigma_i'\longrightarrow p(\Sigma_i')\,\big).
\end{equation}
\noindent We claim that $d_i=1$ for at least one index $i\in\{1,\dots,2^{2g}+1\}$. In fact, suppose on the contrary that $d_i\ge2$ for all $i=1,\dots,2^{2g}+1$. Then we can easily notice that 
\begin{equation}
    d_1+\cdots+d_m=2^{2g}\ge2m \quad  \text{ and } \quad d_{m+1}+\cdots+d_{2^{2g}+1}=2^{2g}\ge 2^{2g+1}-2m+2.
\end{equation}
However, these inequalities put together readily imply the following chain of inequalities
\begin{equation}
    0<2^{2g-1}+1\le m \le 2^{2g-1}
\end{equation}
\noindent which is clearly impossible. Therefore there is at least one component, say $\Sigma_1'$ up to relabelling, such that $d_1=1$. This implies that $p_{|\,\Sigma_1'}$ is a homeomorphism. However, this is impossible, because $\Sigma_1'$ is not a disc and hence it must contain a non-separating simple loop, which is a lift of its image under the covering projection. This leads to the desired contradiction.
\end{proof}



\smallskip

\noindent In order to get the desired group $G$ we apply the argument above one more time. Consider the covering space $q\colon\Sigma\longrightarrow \Sigma'$ corresponding to the kernel of $\pi_1(\Sigma')\longrightarrow \textnormal{H}_1(\Sigma',\,\mathbb Z_2)$. Since no simple non-separating loop on $\Sigma'$ lifts to $\Sigma$, it directly follows that no essential simple loop on $S$ lifts to $\Sigma$ via the covering projection $p\circ q\colon\Sigma\longrightarrow \Sigma'\longrightarrow S$. It remains to show that $p\circ q$ is a regular covering. However, this follows by construction because $p_*\big(\pi_1(\Sigma')\big)$ is characteristic in $\pi_1(S)$ and, in the same fashion, $q_*\big(\pi_1(\Sigma)\big)$ is characteristic in $\pi_1(\Sigma')$. As a direct consequence we have that $(p\circ q)_*(\pi_1(\Sigma))$ is characteristic in $\pi_1(S)$ and hence normal. Thence the covering $p\circ q$ is regular. We consider the representation $\rho$ as above in \eqref{eq:proj}, namely
\begin{equation}\label{eq:proj2}
    \rho\colon\pi_1(S)\longrightarrow \frac{\pi_1(S)}{(p\circ q)_*(\pi_1(\Sigma))}=G.
\end{equation}
We may observe that a closed path in $\pi_1(S)$ is in $\textnormal{ker}(\rho)$ if and only if the path lifts to a closed path in $\Sigma$. However, by construction, no simple closed loop in $S$ lifts to $\Sigma$ and hence the representation $\rho$ has non-geometric kernel as desired.

\smallskip

\section{Realization process}\label{sec:rp}

\noindent In this section we aim to provide the desired counterexample of the simple loop conjecture in every dimension $n\ge4$.

\subsection{Finding a $n-$manifold with prescribed fundamental group} It is a well-known fact that any finitely presented group can be realized as the fundamental group of some $n-$dimensional manifold $M$. We recall here a well-known construction for the sake of completeness. Let $G$ be a finitely presented group with presentation 
\begin{equation}
G\cong\langle \, a_1,\dots,a_k\,|\,r_1,\dots,r_\ell\,\rangle.
\end{equation}

\noindent Let us consider $\mathbb S^n$ and then $k$ copies of $\mathbb S^{n-1}\times\mathbb S^1$ and let $M_o$ be the closed $n-$manifold
\begin{equation}
    M_o=\mathbb S^n\,\#\,(\mathbb S^{n-1}\times\mathbb S^1)\,\#\,\cdots\,\#\,(\mathbb S^{n-1}\times\mathbb S^1).
\end{equation}

\noindent Since $\pi_1(\mathbb S^{n-1}\times\mathbb S^1)\cong\mathbb Z$, it readily follows that $\pi_1(M_o)\cong\langle\,a_1,\dots,a_k \mid \phi\rangle$. This is a consequence of the following Lemma whose proof is an application of Seifert-van Kampen's Theorem.

\begin{lem}\label{lem:prodgroup}
Let $n\ge3$ and let $M$ and $N$ be closed smooth $n-$manifolds. Then 
\begin{equation}
    \pi_1(M\,\#\,N)=\pi_1(M)*\pi_1(N).
\end{equation}
\end{lem}

\smallskip

\noindent In order to realise a manifold $M$ with fundamental group $G$ we proceed as follows. We can represent a generic relation $r$ by some smooth embedded loop, say $\gamma$, in $M_o$. Let $T_\gamma$ be a tubular neighborhood of $\gamma$ homeomorphic to $\mathbb S^1\times\mathbb D^{n-1}$. By cutting out the tube $T_\gamma$ from $M_o$ we obtain a $n-$manifold whose boundary is homeomorphic to $\mathbb S^1\times\mathbb S^{n-2}$. The crucial observation here is the following fact: Both spaces $\mathbb S^1\times \mathbb D^{n-1}$ and $\mathbb S^{n-2}\times\mathbb D^{2}$ have the same boundary, up to homeomorphism, namely $\mathbb S^1\times\mathbb S^{n-2}$. Then, we glue a copy of $\mathbb S^{n-2}\times\mathbb D^{2}$ along the boundary. Since $\mathbb S^{n-2}\times\mathbb D^{2}$ is simply connected -- notice that here we make the use of the assumption $n\ge4$ -- it can be showed that the resulting space has fundamental group isomorphic to $\langle \, a_1,\dots,a_k\,|\,r\,\rangle$. Let us show this fact.

\smallskip

\noindent The only thing one has to check is the following: Removing the tube $T_\gamma$ and replacing it with $\mathbb S^{n-2} \times \mathbb D^2$ has the effect of killing $\gamma$. Let $M_\gamma$ denote the manifold obtained from $M_o$ by this process. We have that
\begin{equation}
    M_\gamma=\Big( M_o \setminus \text{int}(T_\gamma)\Big) \cup_{\mathbb S^1 \times \mathbb S^{n-2}} \big(\,\mathbb D^2 \times \mathbb S^{n-2}\,\big)
\end{equation}

\noindent Notice that the following holds:
\begin{equation}
\pi_1\Big(M_o \setminus \text{int}(T_\gamma)\Big)\cong\pi_1(M_o)=\langle a_1, \dots, a_k \mid \phi\rangle, \\
\end{equation}
\noindent because two loops in an $n-$dimensional manifold do not link each other if $n\ge4$. Moreover we have that:
\begin{equation}
\pi_1(\mathbb D^2 \times \mathbb S^{n-2}) = \{0\}
\end{equation}
\noindent and 
\begin{equation}
\pi_1(\mathbb S^1 \times \mathbb S^{n-2})\cong\Z= \langle\, z \,\rangle,
\end{equation}

\noindent where $z$ is represented by $\mathbb S^1 \times \{\,*\,\}$ in $\mathbb S^1 \times \mathbb S^{n-2}$. Note that the image of $z$ in $M_\gamma$ is precisely the curve $\gamma$. Therefore by Seifert-van Kampen we have that $\pi_1(M_\gamma) = \langle a_1, \dots, a_{k} \mid r \rangle$.

\medskip

\noindent On our manifold $M_o$, we represent each relator $r_i$ with a smooth loop $\gamma_i$ and, without loss of generality, we may assume all these loops pairwise disjoint. Therefore, by applying this process $\ell$ times we eventually obtain a $4-$manifold $M$ with fundamental group $G$ as desired.

\subsection{Finding a continuous map} For a closed genus $g$ surface $S$, let $G$ be the finite group obtained in Section \S\ref{sec:fpg} and let $M$ be a $n-$manifold with fundamental group $\pi_1(M)\cong G$. The representation in \eqref{eq:proj2} can be now regarded as a representation $\rho\colon\pi_1(S)\longrightarrow \pi_1(M)$. Recall that, from Section \S\ref{sec:fpg} we know $\rho$ has non-geometric kernel. Here we aim to find a continuous map $f\colon S\longrightarrow M$ such that $f_*=\rho$. As a result, we obtain the desired counterexample to the simple loop conjecture in dimension $n$.

\smallskip

\noindent The existence of a mapping $f\colon S\longrightarrow M$ now becomes an obstruction theoretic matter. Notice that both $S$ and $M$ admit a CW structure and we can choose these structures to have only one $0-$cell. Being $S$ a closed genus $g$ surface, we consider it as a CW complex with one $0-$cell, $2g$ $1-$cells and only one $2-$cell. We use $\rho$ to define a mapping $f$ at the level of $1-$skeleton by sending every $1-$cell, say $e$, to a cellular representative of $\rho(e)$. Since $\rho$ is a map of fundamental groups it respects homotopies between paths and thence we can extend $f$ to the $2-$skeleton. In fact, the boundary of the unique $2-$cell is trivial in $\pi_1(S)$ so it is mapped to a trivial loop in $\pi_1(M)$. Therefore there is a continuous map $f\colon S\longrightarrow M$ that realises $\rho$ as a map of fundamental groups as desired. It remains to show that $f$ is $2-$sided. This is the case: In fact, since $M$ is orientable its orientation character $\chi_M$ is trivial and, in particular, $\chi_M \circ f_*(\gamma)=0$ for all $\gamma\in\pi_1(S)$ and hence $\chi_S=\chi_M\circ f_*$ as desired because $S$ is also orientable.

\subsection{Final remark} The construction above involved orientable manifolds both in the domain and target. It is worth mentioning that, according to Definition \ref{def:sidesmap} a continuous map from a non-orientable surface $S$ to an orientable target $M$ is never $2-$sided. However, a continuous map from an orientable surface $S$ may well be $2-$sided even if the target manifold $M$ is not orientable. Let us give an example of this phenomenon. Let $\rho\colon\pi_1(S)\longrightarrow G$ be a representation into a finitely presented group and let $M$ be an orientable $n-$manifold with fundamental group $G$. We have seen above how to realise a $2-$sided map $f\colon S\longrightarrow M$. Consider $M\,\#\,(\mathbb {RP}^2\,\times\,\mathbb S^{n-2})$; according to Lemma \ref{lem:prodgroup} its fundamental group is isomorphic to $G*\mathbb Z_2$. We now realise a $2-$sided map $f\colon S\longrightarrow M\,\#\,(\mathbb {RP}^2\,\times\,\mathbb S^{n-2})$ as follows. Let $M^\circ$ be defined as the open sub-manifold $M\setminus \mathbb D^n$ and notice that $\pi_1(M^\circ)\cong\pi_1(M)\cong G$. Then we first realise a $2-$sided map $f^\circ\colon S\longrightarrow M^\circ$ exactly as above and then define $f=\imath\circ f^\circ$ where $\imath\colon M^\circ\hookrightarrow M\,\#\,(\mathbb {RP}^2\,\times\,\mathbb S^{n-2})$ is an embedding.

\bibliographystyle{amsalpha}
\bibliography{ceslchd}

\end{document}